\documentclass[reqno]{amsart}
\usepackage{amsmath, amssymb, amsthm, amsfonts}
\usepackage{hyperref}

\author{Alexei Yu. Karlovich}
\address{
Centro de Matem\'atica e Aplica\c{c}\~oes (CMA) and
Departamento de Matem\'atica,
Faculdade de Ci\^encias e Tecnologia,
Universidade Nova de Lisboa,
Quinta da Torre,
2829--516 Caparica,
Portugal}
\email{oyk@fct.unl.pt}

\thanks{This work was partially supported by the Funda\c{c}\~ao para a Ci\^encia e a
Tecnologia (Portu\-guese Foundation for Science and Technology)
through the project UID/MAT/00297/2013 (Centro de Matem\'atica e Aplica\c{c}\~oes).}



\title[Hardy-Littlewood maximal operator]
{Hardy-Littlewood Maximal Operator on the Associate Space of a 
Banach Function Space}

 \newtheorem{theorem}{Theorem}
 \newtheorem{lemma}[theorem]{Lemma}
  
 \newtheorem{corollary}[theorem]{Corollary}
 \theoremstyle{remark}
 \newtheorem{definition}[theorem]{Definition}

\numberwithin{equation}{section}

\newcommand{\N}{\mathbb{N}}
\newcommand{\R}{\mathbb{R}}
\newcommand{\Z}{\mathbb{Z}}

\newcommand{\cA}{\mathcal{A}}
\newcommand{\cB}{\mathcal{B}}
\newcommand{\cD}{\mathcal{D}}
\newcommand{\cE}{\mathcal{E}}

\newcommand{\eps}{\varepsilon}
\renewcommand{\kappa}{\varkappa}
\begin{document}

\begin{abstract}
Let $\mathcal{E}(X,d,\mu)$ be a Banach function space over a space of 
homogeneous type $(X,d,\mu)$. We show that if the Hardy-Littlewood maximal 
operator $M$ is  bounded on the space $\mathcal{E}(X,d,\mu)$, then its 
boundedness on the associate space $\mathcal{E}'(X,d,\mu)$ is equivalent to 
a certain condition $\mathcal{A}_\infty$. This result extends a theorem by 
Andrei Lerner from the Euclidean setting of $\mathbb{R}^n$ to the setting of 
spaces of homogeneous type.
\end{abstract}
\maketitle
\section{Introduction.}
We begin with the definition of a space of homogeneous type (see, e.g., 
\cite{C90a}). Given a set $X$ and a function $d:X\times X\to[0,\infty)$, one 
says that $(X,d)$ is a quasi-metric space if the following axioms hold:
\begin{enumerate}
\item[(a)] $d(x,y)=0$ if and only if $x=y$;
\item[(b)] $d(x,y)=d(y,x)$ for all $x,y\in X$;
\item[(c)]  for all $x,y,z\in X$ and some 
constant $\kappa\ge 1$,
\begin{equation}\label{eq:quasi-triangle}
d(x,y)\le \kappa(d(x,y)+d(y,z)).
\end{equation}
\end{enumerate}
For $x\in X$ and $r>0$, consider the ball $B(x,r)=\{y\in X:d(x,y)<r\}$
centered at $x$ of radius $r$. Given a quasi-metric space $(X,d)$ and a 
positive measure $\mu$ that is defined on the $\sigma$-algebra generated by
quasi-metric balls, one says that $(X,d,\mu)$ is a space of homogeneous type
if there exists a constant $C_\mu\ge 1$ such that for any $x\in X$ and any 
$r>0$,
\begin{equation}\label{eq:doubling-measure}
\mu(B(x,2r))\le C_\mu \mu(B(x,r)).
\end{equation}
To avoid trivial measures, we will always assume that $0<\mu(B)<\infty$
for every ball $B$. Consequently, $\mu$ is a $\sigma$-finite measure.

Given a complex-valued function $f\in L^1_{\rm loc}(X,d,\mu)$, we define its 
Hardy-Littlewood maximal function $Mf$ by 
\[
(Mf)(x):=\sup_{B\ni x}\frac{1}{\mu(B)}\int_B |f(x)|\,d\mu(x),
\quad x\in X,
\]
where the supremum is taken over all balls $B\subset X$ containing $x\in X$.
The Hardy-Littlewood maximal operator $M$ is a sublinear operator acting by
the rule $f\mapsto Mf$. The aim of this paper is the studying the relations
between the boundedness of the operator $M$ on a so-called Banach function
space $\cE(X,d,\mu)$ and on its associate space $\cE'(X,d,\mu)$ in the setting
of general spaces of homogeneous type $(X,d,\mu)$.

Let us recall the definition of a Banach function space (see, e.g., 
\cite[Chap.~1, Definition~1.1]{BS88}).
Let $L^0(X,d,\mu)$ denote the set of all complex-valued measurable functions
on $X$ and let $L_+^0(X,d,\mu)$ be the set of all non-negative measurable 
functions on $X$. The characteristic function of a set $E\subset X$ is denoted 
by $\chi_E$. A mapping $\rho: L_+^0(X,d,\mu)\to [0,\infty]$ is called 
a Banach function norm if, for all functions $f,g, f_n\in L_+^0(X,d,\mu)$ 
with $n\in\N$, for all constants $a\ge 0$, and for all measurable subsets 
$E$ of $X$, the following  properties hold:
\begin{eqnarray*}
{\rm (A1)} & &
\rho(f)=0  \Leftrightarrow  f=0\ \mbox{a.e.},
\
\rho(af)=a\rho(f),
\
\rho(f+g) \le \rho(f)+\rho(g),\\
{\rm (A2)} & &0\le g \le f \ \mbox{a.e.} \ \Rightarrow \ 
\rho(g) \le \rho(f)
\quad\mbox{(the lattice property)},\\
{\rm (A3)} & &0\le f_n \uparrow f \ \mbox{a.e.} \ \Rightarrow \
       \rho(f_n) \uparrow \rho(f)\quad\mbox{(the Fatou property)},\\
{\rm (A4)} & & \mu(E)<\infty\ \Rightarrow\ \rho(\chi_E) <\infty,\\
{\rm (A5)} & &\int_E f(x)\,d\mu(x) \le C_E\rho(f)
\end{eqnarray*}
with a constant $C_E \in (0,\infty)$ that may depend on $E$ and 
$\rho$,  but is independent of $f$. When functions differing only on 
a set of measure  zero are identified, the set $\cE(X,d,\mu)$ of all 
functions $f\in L^0(X,d,\mu)$ for  which  $\rho(|f|)<\infty$ is called a 
Banach function space. For each $f\in \cE(X,d,\mu)$, the norm of $f$ is 
defined by 
\[
\|f\|_\cE :=\rho(|f|). 
\]
The set $\cE(X,d,\mu)$ under the natural 
linear space operations and under  this norm becomes a Banach space (see 
\cite[Chap.~1, Theorems~1.4 and~1.6]{BS88}). 
If $\rho$ is a Banach function norm, its associate norm 
$\rho'$ is defined on $L_+^0(X,d,\mu)$ by
\[
\rho'(g):=\sup\left\{
\int_X f(x)g(x)\,d\mu(x) \ : \ 
f\in L_+^0(X,d,\mu), \ \rho(f) \le 1
\right\}.
\]
It is a Banach function norm itself \cite[Chap.~1, Theorem~2.2]{BS88}.
The Banach function space $\cE'(X,d,\mu)$ determined by the Banach function 
norm $\rho'$ is called the associate space (K\"othe dual) of $\cE(X,d,\mu)$. 

Hyt\"onen and Kairema \cite{HK12}, developing further ideas of Christ
\cite{C90a}, show that a space of homogeneous type $(X,d,\mu)$ can be equipped
with a finite system of adjacent dyadic grids $\{\cD^t:t=1,\dots,K\}$,
each of which consists of sets $Q$, called dyadic cubes, that resemble
properties of usual dyadic cubes in $\mathbb{R}^n$. We postpone precise 
formulations of these results until Section~\ref{sec:dyadic-decomposition}.

Given a dyadic grid  $\cD\in\bigcup_{t=1}^K \cD^t$, a sparse family 
$S\subset\cD$ is a collection  of dyadic cubes $Q\in\cD$ for which there 
exists a collection of sets $\{E(Q)\}_{Q\in S}$ such that the sets $E(Q)$ are 
pairwise disjoint, $E(Q)\subset Q$, and 
\[
\mu(Q)\le 2\mu(E(Q)).
\]
\begin{definition}[The condition $\cA_\infty$]
\label{def:A-infinity}
Following \cite{L17},
we say that a Banach function space $\cE(X,d,\mu)$ over a space of homogeneous 
type $(X,d,\mu)$ satisfies the condition $\cA_\infty$ if there exist constants 
$C,\gamma>0$ such that for every dyadic grid $\cD\in\bigcup_{t=1}^K\cD^t$,
every finite sparse family $S\subset\cD$, every collection of non-negative
numbers $\{\alpha_Q\}_{Q\in S}$, and every collection of pairwise disjoint
measurable sets $\{G_Q\}_{Q\in S}$ such that $G_Q\subset Q$, one has
\begin{equation}\label{eq:A-infty}
\Bigg\|\sum_{Q\in S}\alpha_Q\chi_{G_Q}\Bigg\|_\cE
\le 
C\left(\max_{Q\in S}\frac{\mu(G_Q)}{\mu(Q)}\right)^\gamma
\Bigg\|\sum_{Q\in S}\alpha_Q\chi_Q\Bigg\|_\cE.
\end{equation}
\end{definition}
The main aim of the present paper is to provide a self-contained proof of
the following generalization of \cite[Theorem~3.1]{L17} from the Euclidean
setting of $\R^n$ to the setting of spaces of homogeneous type.
\begin{theorem}[Main result]
\label{th:main}
Let $\cE(X,d,\mu)$ be a Banach function space over a space of homogeneous type 
$(X,d,\mu)$ and let $\cE'(X,d,\mu)$ be its associate space. 
\begin{enumerate}
\item[{\rm(a)}]
If the Hardy-Littlewood maximal operator $M$ is bounded on the space
$\cE'(X,d,\mu)$, then the space $\cE(X,d,\mu)$ satisfies the condition 
$\cA_\infty$.

\item[{\rm(b)}]
If the Hardy-Littlewood maximal operator $M$ is bounded on the space 
$\cE(X,d,\mu)$ and the space $\cE(X,d,\mu)$ satisfies the condition 
$\cA_\infty$, then the operator $M$ is bounded on the space $\cE'(X,d,\mu)$.
\end{enumerate}
\end{theorem}
The paper is organized as follows. In Section~\ref{sec:dyadic-decomposition},
we formulate results of Hyt\"onen and Kairema \cite{HK12} on a construction
of a system of adjacent dyadic grids $\{\cD^t:t=1,\dots,K\}$ on the underlying
space of homogeneous type $(X,d,\mu)$. In 
Section~\ref{sec:reverse-Hoelder-inequality}, we prove that if a weight $w$ 
belongs to the dyadic class $A_1^\cD$ with $\cD\in\bigcup_{t=1}^K\cD^t$, then 
it satisfies a reverse H\"older inequality. 
Section~\ref{sec:Calderon-Zygmund-decomposition} contains a proof of a version 
of the Calder\'on-Zygmund decomposition of spaces of homogeneous type. Armed 
with these auxiliary results, following ideas of Lerner \cite[Theorem~3.1]{L17},
we give a self-contained proof of Theorem~\ref{th:main} in  
Sections~\ref{sec:proof-necessity} and \ref{sec:proof-sufficiency}.

We conjecture that reflexive variable Lebesgue spaces 
(see, e.g., \cite{CDF09,CS18,D05, L17} for definitions)
over spaces of homogeneous type 
satisfy the condition $\cA_\infty$. If this conjecture is true, then
in view of Theorem~\ref{th:main}, we can affirm that the Hardy-Littlewood 
maximal operator is bounded on a reflexive variable Lebesgue
space over a space of homogeneous type if and only if it is bounded on its 
associate space. Note that in the Euclidean setting of $\R^n$, this result
was proved by Diening \cite[Theorem~8.1]{D05} (see also 
\cite[Theorem~1.1]{L17}). We are going to embark on the proof of the above 
conjecture in a forthcoming paper.
\section{Dyadic decomposition of spaces of homogeneous type.}
\label{sec:dyadic-decomposition}
Let $(X,d,\mu)$ be a space of homogeneous type. The doubling property of $\mu$ 
implies the following geometric doubling property of the quasi-metric $d$:
any ball $B(x,r)$ can be covered by at most $N:=N(C_\mu,\kappa)$ balls of 
radius $r/2$. It is not difficult to show that $N\le C_\mu^{6+3\log_2\kappa}$.

An important tool for our proofs is the concepts of an adjacent system of 
dyadic grids $\cD^t$, $t\in\{1,\dots,K\}$, on  a space of homogeneous type 
$(X,d,\mu)$. Christ \cite[Theorem~11]{C90a} (see also 
\cite[Chap.~VI, Theorem~14]{C90b}) 
constructed a system of sets on  $(X,d,\mu)$, which satisfy many of the 
properties of a system of dyadic cubes on the Euclidean space. His 
construction was further refined by Hyt\"onen and 
Kairema  \cite[Theorem~2.2]{HK12}. We will use the version from 
\cite[Theorem~4.1]{AHT17}.
\begin{theorem}\label{th:Hytonen-Kairema}
Let $(X,d,\mu)$ be a space of homogeneous type with the constant $\kappa\ge 1$
in inequality \eqref{eq:quasi-triangle} and the geometric doubling constant $N$.
Suppose the parameter $\delta\in(0,1)$ satisfies $96\kappa^2\delta\le 1$. Then 
there exist an integer number $K=K(\kappa,N,\delta)$, a countable set of points
$\{z_\alpha^{k,t}:\alpha\in\cA_k\}$ with $k\in\Z$ and  $t\in\{1,\dots,K\}$, 
and a finite number of dyadic grids 
\[
\cD^t:=\{Q_\alpha^{k,t}:k\in\Z,\alpha\in\cA_k\}, 
\]
such that the following properties are fulfilled:
\begin{enumerate}
\item[{\rm(a)}]
for every $t\in\{1,\dots,K\}$ and $k\in\Z$ one has
\begin{enumerate}
\item[{\rm (i)}]
$X=\bigcup_{\alpha\in\cA_k} Q_\alpha^{k,t}$ (disjoint union);

\item[{\rm (ii)}]
if $Q,P\in\cD^t$, then $Q\cap P\in\{\emptyset, Q,P\}$;

\item[{\rm (iii)}]
if $Q_\alpha^{k,t}\in\cD^t$, then
\begin{equation}\label{eq:dyadic-cubes}
B(z_\alpha^{k,t},c_1\delta^k)
\subset 
Q_\alpha^{k,t}
\subset 
B(z_\alpha^{k,t},C_1\delta^k),
\end{equation}
where $c_1=(12\kappa^4)^{-1}$ and $C_1:=4\kappa^2$;
\end{enumerate}

\item[{\rm(b)}]
for every $t\in\{1,\dots,K\}$ and every $k\in\Z$, if $Q_\alpha^{k,t}\in\cD^t$, 
then there exists at least one $Q_\beta^{k+1,t}\in\cD^t$, which is called a 
child of $Q_\alpha^{k,t}$, such that  $Q_\beta^{k+1,t}\subset Q_\alpha^{k,t}$, 
and there exists exactly one $Q_\gamma^{k-1,t}\in\cD^t$, which is 
called the parent of $Q_\alpha^{k,t}$, such that 
$Q_\alpha^{k,t}\subset Q_\gamma^{k-1,t}$;

\item[{\rm (c)}]
for every ball $B=B(x,r)$, there exists 
\[
Q_B\in\bigcup_{t=1}^K\cD^t
\]
such that $B\subset Q_B$ and $Q_B=Q_\alpha^{k-1,t}$ for some indices 
$\alpha\in\cA_k$ and $t\in\{1,\dots,K\}$, where $k$ is 
the unique integer such that $\delta^{k+1}<r\le\delta^k$.
\end{enumerate}
\end{theorem}
The collections $\cD^t$, $t\in\{1,\dots,K\}$, are called dyadic grids on $X$. 
The sets $Q_\alpha^{k,t}\in\cD^t$ are referred to as dyadic cubes with center 
$z_\alpha^{k,t}$ and sidelength $\delta^k$, see \eqref{eq:dyadic-cubes}.
The sidelength of a cube $Q\in\cD^t$ will be denoted by $\ell(Q)$.
We should emphasize that these sets are not cubes in the standard sense even 
if the underlying space is $\R^n$. Parts (a) and (b) of the above theorem  
describe dyadic grids $\cD^t$, with $t\in\{1,\dots,K\}$, individually. 
In particular, 
\eqref{eq:dyadic-cubes} permits a comparison between a dyadic cube and
quasi-metric balls. Part (c) guarantees the existence of a finite family
of dyadic grids such that an arbitrary quasi-metric ball is contained in a 
dyadic cube in one of these grids. Such a finite family of dyadic grids is 
referred to as an adjacent system of dyadic grids. 

Let $\cD\in\bigcup_{t=1}^K \cD^t$ be a fixed dyadic grid. One can define the 
dyadic maximal function $M^\cD f$ of a function $f\in L_{\rm loc}^1(X,d,\mu)$
by
\[
(M^\cD f)(x)=\sup_{Q\ni x}\frac{1}{\mu(Q)}\int_Q |f(x)|\,d\mu(x),
\quad x\in X,
\]
where the supremum is taken over all dyadic cubes $Q\in\cD$ containing $x$.

The following important result is proved by Hyt\"onen and Kairema 
\cite[Proposition~7.9]{HK12}.
\begin{theorem}\label{th:HK-pointwise}
Let $(X,d,\mu)$ be a space of homogeneous type and let $\bigcup_{t=1}^K \cD^t$ 
be the adjacent system of dyadic grids given by Theorem~\ref{th:Hytonen-Kairema}.
There exist a constant $C_{HK}(X)\ge 1$ depending only $(X,d,\mu)$ such that 
for every $f\in L_{\rm loc}^1(X,d,\mu)$ and a.e. $x\in X$, one has
\begin{align*}
(M^{\cD^t} f)(x)
&\le 
C_{HK}(X) (Mf)(x),\quad t\in\{1,\dots,K\},
\\
(Mf)(x)
&\le 
C_{HK}(X)\sum_{t=1}^K (M^{\cD^t}f)(x).
\end{align*}
\end{theorem}
\section{Reverse H\"older inequality.}
\label{sec:reverse-Hoelder-inequality}
A measurable non-negative locally integrable function $w$ on $X$ is said to
be a weight. Given a weight $w$ and a measurable set
$E\subset X$, denote
\[
w(E):=\int_E w(x)\,d\mu(x).
\]

Fix a dyadic grid $\cD\in\bigcup_{t=1}^K\cD^t$.
A weight $w:X\to[0,\infty]$ is said to belong to the dyadic class
$A_1^\cD$ if there exists a constant $c>0$ such that for a.e. $x\in X$,
\[
(M^\cD w)(x)\le cw(x).
\]
The smallest constant $c$ in this inequality is denoted by $[w]_{A_1^\cD}$.

Following \cite[Definition~4.4]{AHT17}, a generalized dyadic parent (gdp) of
a cube $Q$ is any cube $Q^*$ such that $\ell(Q^*)=\frac{1}{\delta^2}\ell(Q)$
and for every $Q'\in\cD$ such that $Q'\cap Q\ne\emptyset$ and $\ell(Q')=\ell(Q)$,
one has $Q'\subset Q^*$. According to \cite[Lemma~4.5]{AHT17}, every
cube $Q\in\cD$ possesses at least one gdp. 

For every $x\in X$ and $Q\in\cD$, put
\begin{align*}
\mathcal{Q}_Q:=\{Q'\in\cD\ :\ Q'\cap Q\ne\emptyset,\ \ell(Q')\le \ell(Q)\},
\
\mathcal{Q}_Q^x:=\{Q'\in\mathcal{Q}_Q\ :\ x\in Q'\}.
\end{align*}
It follows immediately that if $Q'\in\mathcal{Q}_Q$, then $Q'\subset Q^*$.
For every $Q\in\cD$, the localized dyadic maximal operator $M_Q$ is defined
by
\[
(M_Q f)(x)=\left\{
\begin{array}{lll}
\displaystyle\sup_{Q'\in\mathcal{Q}_Q^x}\frac{1}{\mu(Q')}\int_{Q'} |f(y)|\,d\mu(y)
&\mbox{if}& \mathcal{Q}_Q^x\ne\emptyset,
\\
0, &\mbox{if}& \mathcal{Q}_Q^x=\emptyset.
\end{array}\right.
\]
Following \cite[Definition~4.7]{AHT17}, one says that a weight 
$w:X\to[0,\infty]$ belongs to the dyadic class $A_\infty^\cD$ if
\[
[w]_\infty^\cD
:=
\sup_{Q\in\cD}\inf_{Q^*}\frac{1}{w(Q^*)}\int_X (M_Qw)(x)\,d\mu(x)<\infty.
\]

Let $C_\cD\ge 1$ be a constant such that for all cubes $Q\in\cD$ and 
$Q'\in\mathcal{Q}_Q$ satisfying $\ell(Q)=\ell(Q')$, one has 
\begin{equation}\label{eq:constant-CD}
\mu(Q^*)\le C_\cD\mu(Q').
\end{equation}
\begin{lemma}\label{le:A1-A-infinity}
If $w\in A_1^\cD$, then $w\in A_\infty^\cD$ and 
$[w]_{A_\infty^\cD}\le[w]_{A_1^\cD}$.
\end{lemma}
\begin{proof}
Fix a cube $Q\in\cD$ and one of its gdp's $Q^*$. It follows immediately from
the definition of $\mathcal{Q}_Q$ that if $Q'\in\mathcal{Q}_Q$, then
$Q'\subset Q^*$. Take any $x\in X$. If $\mathcal{Q}_Q^x\ne\emptyset$, then
there exists $Q'\in\mathcal{Q}_Q^x$ such that $x\in Q'\subset Q^*$.
Therefore, if $x\notin Q^*$, then $(M_Qw)(x)=0$. Thus, for a.e. $x\in Q$,
\[
(M_Qw)(x)=(M_Qw)(x)\chi_{Q^*}(x)
\le 
(M^\cD w)(x)\chi_{Q^*}(x)
\le 
[w]_{A_1^\cD} w(x)\chi_{Q^*}(x),
\]
whence
\begin{align*}
[w]_{A_\infty^\cD}
&=
\sup_{Q\in\cD}\inf_{Q^*}\frac{1}{w(Q^*)}\int_X (M_Q w)(x)\,d\mu(x) 
\\
&\le 
[w]_{A_1^\cD}\sup_{Q\in\cD}\inf_{Q^*}\frac{1}{w(Q^*)}\int_{Q^*}w(x)\,d\mu(x)
=[w]_{A_1^\cD},
\end{align*}
which completes the proof.
\end{proof}
The following result is an easy consequence of the weak reverse H\"older
inequality for weights in $A_\infty^\cD$ obtained recently
by Anderson, Hyt\"onen, and Tapiola \cite[Theorem~5.4]{AHT17}.
\begin{lemma}\label{le:RHI}
Let $K$ be the constant from Theorem~\ref{th:Hytonen-Kairema} and
$C_\cD$ be the constant defined in \eqref{eq:constant-CD}.
If $w\in A_1^\cD$, then for every $\eta$ satisfying
\begin{equation}\label{eq:RHI-1}
0<\eta\le\frac{1}{2 C_\cD^2 K[w]_{A_1^\cD}}
\end{equation}
and every $Q\in\cD$, one has
\begin{equation}\label{eq:RHI-2}
\left(\frac{1}{2\mu(Q)}\int_Q w^{1+\eta}(x)\,d\mu(x)\right)^{\frac{1}{1+\eta}}
\le 
C_\cD[w]_{A_1^\cD}\frac{1}{\mu(Q)}\int_Q w(x)\,d\mu(x).
\end{equation}
\end{lemma}
\begin{proof}
We know from Lemma~\ref{le:A1-A-infinity} that $w\in A_\infty^\cD$ and
$[w]_{A_\infty^\cD}\le [w]_{A_1^\cD}$. Then, by \cite[Theorem~5.4]{AHT17},
for every $\eta$ satisfying \eqref{eq:RHI-1}, one has
\begin{equation}\label{eq:RHI-3}
\left(\frac{1}{2\mu(Q)}\int_Q w^{1+\eta}(x)\,d\mu(x)\right)^{\frac{1}{1+\eta}}
\le 
C_\cD\frac{1}{\mu(Q^*)}\int_{Q^*}w(x)\,d\mu(x).
\end{equation}
Since $w\in A_1^\cD$, for a.e. $x\in Q\subset Q^*$, one has
\[
\frac{1}{\mu(Q^*)}\int_{Q^*} w(y)\,d\mu(y)
\le 
(M^\cD w)(x)
\le 
[w]_{A_1^\cD}w(x).
\]
Integrating this inequality over $Q$, we obtain
\begin{equation}\label{eq:RHI-4}
\frac{\mu(Q)}{\mu(Q^*)}\int_{Q^*}w(y)\,d\mu(y)
\le 
[w]_{A_1^\cD}\int_Q w(x)\,d\mu(x).
\end{equation}
Combining inequalities \eqref{eq:RHI-3} and \eqref{eq:RHI-4}, we immediately 
arrive at inequality \eqref{eq:RHI-2}.
\end{proof}
The main result of this section is the following reverse H\"older inequality.
\begin{theorem}\label{th:RHI-corollary}
Let $K$ be the constant from Theorem~\ref{th:Hytonen-Kairema} and
$C_\cD$ be the constant defined in \eqref{eq:constant-CD}.
If $w\in A_1^\cD$, then for every $\eta$ satisfying \eqref{eq:RHI-1},
every cube $Q\in\cD$, and every measurable subset $E\subset Q$, one has
\begin{equation}\label{eq:RHI-corollary}
\frac{w(E)}{w(Q)}
\le 
2^{\frac{1}{1+\eta}}C_\cD[w]_{A_1^\cD}
\left(\frac{\mu(E)}{\mu(Q)}\right)^{\frac{\eta}{1+\eta}}.
\end{equation}
\end{theorem}
\begin{proof}
By H\"older's inequality and reverse H\"older's inequality \eqref{eq:RHI-2},
\begin{align*}
w(E)
&=
\int_Q w(x)\chi_E (x)\,d\mu(x)
\\
&\le
\left(\int_Q w^{1+\eta}(x)\,d\mu(x)\right)^{\frac{1}{1+\eta}}
\big(\mu(E)\big)^{\frac{\eta}{1+\eta}}
\\
&\le 
2^{\frac{1}{1+\eta}}C_\cD[w]_{A_1^\cD}
\frac{\big(\mu(Q)\big)^{\frac{1}{1+\eta}}}{\mu(Q)}w(Q)
\big(\mu(E)\big)^{\frac{\eta}{1+\eta}}
\\
&=
2^{\frac{1}{1+\eta}}C_\cD[w]_{A_1^\cD}w(Q)
\left(\frac{\mu(E)}{\mu(Q)}\right)^{\frac{\eta}{1+\eta}},
\end{align*}
which immediately implies \eqref{eq:RHI-corollary}.
\end{proof}
\section{Calder\'on-Zygmund decomposition.}
\label{sec:Calderon-Zygmund-decomposition}
We start this section with the following important observation.
\begin{lemma}\label{le:measure-child-father}
Suppose $(X,d,\mu)$ is a space of homogeneous type with the constants
$\kappa\ge 1$ in inequality \eqref{eq:quasi-triangle} and $C_\mu\ge 1$ in 
inequality \eqref{eq:doubling-measure}. Let $(X,d,\mu)$ be equipped with 
an adjacent system of dyadic grids $\{\cD^t,t=1,\dots,K\}$ and let 
$\delta\in(0,1)$ be chosen as in Theorem~\ref{th:Hytonen-Kairema}. Then there 
is an $\eps=\eps(\kappa,C_\mu,\delta)\in(0,1)$ such that for every 
$t\in\{1,\dots,K\}$ and all $Q,P\in\cD^t$, if $Q$ is a child of $P$, then 
\begin{equation}\label{eq:measure-child-father-1}
\mu(Q)\ge\eps\mu(P).
\end{equation}
\end{lemma}
This result is certainly known. For the construction of Christ, we refer to
\cite[Chap.~VI, Theorem 14]{C90b}, where it is stated without proof (see also 
\cite[Theorem~11]{C90a}, where it is implicit). In \cite[Theorem~2.1]{ACM15}
and \cite[Theorem~2.5]{CS18} it is stated without proof and attributed to 
Hyt\"onen and Kairema \cite{HK12}, although it is only implicit in the latter 
paper. For the convenience of the readers, we provide its proof.
\begin{proof}
Let $Q=Q_\beta^{k+1,t}$ be a child of $P=Q_\alpha^{k,t}$ for some
$t\in\{1,\dots,K\}$, $k\in\Z$, and $\alpha\in\cA_k$, $\beta\in\cA_{k+1}$.
It follows from Theorem~\ref{th:Hytonen-Kairema}(a), part (iii), that
$P\subset B(z_\alpha^{k,t},4\kappa^2\delta^k)$ and 
$B(z_\beta^{k+1,t},(12\kappa^4)^{-1}\delta^{k+1})\subset Q$, whence
\begin{equation}\label{eq:measure-child-father-2}
\mu(P)\le\mu\big(B(z_\alpha^{k,t},4\kappa^2\delta^k)\big),
\quad
\mu\big(B(z_\beta^{k+1,t},(12\kappa^4)^{-1}\delta^{k+1})\big)\le\mu(Q).
\end{equation}
It follows from \cite[Lemma~2.10]{HK12} with $C_0=2\kappa$ 
(cf. \cite[Lemma~4.10]{HK12}) that if $Q_\beta^{k+1,t}$ is a child of
$Q_\alpha^{k,t}$, then
\begin{equation}\label{eq:measure-child-father-3}
d(z_\alpha^{k,t},z_\beta^{k+1,t})<2\kappa\delta^k.
\end{equation}
If $x\in B(z_\alpha^{k,t},4\kappa^2\delta^k)$, then 
\begin{equation}\label{eq:measure-child-father-4}
d(x,z_\alpha^{k,t})<4\kappa^2\delta^k.
\end{equation}
Combining \eqref{eq:quasi-triangle} with 
\eqref{eq:measure-child-father-3}--\eqref{eq:measure-child-father-4}, we get
\begin{align*}
d(x,z_\beta^{k+1,t})
&\le 
\kappa\big(d(x,z_\alpha^{k,t})+d(z_\beta^{k+1,t},z_\alpha^{k,t})\big)
\\
&< \kappa(4\kappa^2\delta^k+2\kappa\delta^k)=\kappa^2(4\kappa+2)\delta^k,
\end{align*}
whence $x\in B(z_\beta^{k+1,t},\kappa^2(4\kappa+2)\delta^k)$. Therefore
\[
B(z_\alpha^{k,t},4\kappa^2\delta^k)
\subset
B(z_\beta^{k+1,t},\kappa^2(4\kappa+2)\delta^k).
\]
This inclusion immediately implies that
\begin{equation}\label{eq:measure-child-father-5}
\mu\big(B(z_\alpha^{k,t},4\kappa^2\delta^k)\big)
\leq 
\mu\big(B(z_\beta^{k+1,t},\kappa^2(4\kappa+2)\delta^k)\big).
\end{equation}

Let $s$ be the smallest natural number satisfying 
\[
\log_2(12\kappa^6(4\kappa+2)\delta^{-1})\le s.
\]
Then $\kappa^2(4\kappa+2)\delta^k\le 2^s(12\kappa^4)^{-1}\delta^{k+1}$ and,
therefore,
\begin{equation}\label{eq:measure-child-father-6}
\mu\big(B(z_\beta^{k+1,t},\kappa^2(4\kappa+2)\delta^k)\big)
\leq 
\mu\big(B(z_\beta^{k+1,t},2^s(12\kappa^4)^{-1}\delta^{k+1})\big).
\end{equation}
Applying inequality \eqref{eq:doubling-measure} $s$ times, one gets
\begin{equation}\label{eq:measure-child-father-7}
\mu\big(B(z_\beta^{k+1,t},2^s(12\kappa^4)^{-1}\delta^{k+1})\big)
\leq
C_\mu^s\mu\big(B(z_\beta^{k+1,t},(12\kappa^4)^{-1}\delta^{k+1})\big).
\end{equation}
Combining inequalities \eqref{eq:measure-child-father-2} with
\eqref{eq:measure-child-father-5}--\eqref{eq:measure-child-father-7}, 
we arrive at
\begin{align*}
\mu(P)
& \le 
\mu\big(B(z_\beta^{k+1,t},\kappa^2(4\kappa+2)\delta^k\big) 
\\
&\le 
C_\mu^s\mu\big(B(z_\beta^{k+1,t}(12\kappa^4)^{-1}\delta^{k+1}\big) 
\le 
C_\mu^s\mu(Q),
\end{align*}
which implies inequality \eqref{eq:measure-child-father-1} with 
$\eps=C_\mu^{-s}$.
\end{proof}
Once Lemma~\ref{le:measure-child-father} is available, one can prove the 
following version of the Calder\'on-Zygmund decomposition for spaces of
homogeneous type.
\begin{theorem}
\label{th:CZ-decomposition}
Let $(X,d,\mu)$ be a space of homogeneous type and 
$\cD\in\bigcup_{t=1}^K\cD_t$ be a dyadic grid. Suppose that $\eps\in(0,1)$ is 
the same as in Lemma~\ref{le:measure-child-father} and $f\in L^1(X,d,\mu)$.
\begin{enumerate}
\item[{\rm(a)}]
If 
\[
\lambda>\left\{\begin{array}{lll}
0 &\mbox{if} &\mu(X)=\infty,
\\[3mm]
\displaystyle
\frac{1}{\mu(X)}\int_X|f(x)|\,d\mu(x) &\mbox{if}& \mu(X)<\infty,
\end{array}\right.
\]
and the set
\[
\Omega_\lambda:=\{x\in X\ :\ (M^\cD f)(x)>\lambda\}
\]
is nonempty, then there exists a collection $\{Q_j\}\subset\cD$ that is
pairwise disjoint, maximal with respect to inclusion, and such that
\begin{equation}\label{eq:CZ-decomposition-1}
\Omega_\lambda=\bigcup_j Q_j.
\end{equation}
Moreover, for every $j$,
\begin{equation}\label{eq:CZ-decomposition-2}
\lambda<\frac{1}{\mu(Q_j)}\int_{Q_j}|f(x)|\,d\mu(x)\le\frac{\lambda}{\eps}.
\end{equation}

\item[{\rm(b)}]
Let $a>{2}/{\eps}$ and, for $k\in\Z$ satisfying 
\begin{equation}\label{eq:CZ-decomposition-3}
a^k>\left\{\begin{array}{lll}
0 &\mbox{if} &\mu(X)=\infty,
\\[3mm]
\displaystyle
\frac{1}{\mu(X)}\int_X|f(x)|\,d\mu(x) &\mbox{if}& \mu(X)<\infty,
\end{array}\right.
\end{equation}
let
\begin{equation}\label{eq:CZ-decomposition-4}
\Omega_k:=\{x\in X\ :\ (M^\cD f)(x)>a^k\}.
\end{equation}
If $\Omega_k\ne\emptyset$, then there exists a collection 
$\{Q_j^k\}_{j\in J_k}$ (as in part {\rm (a)}) such that it is pairwise
disjoint, maximal with respect to inclusion, and
\begin{equation}\label{eq:CZ-decomposition-5}
\Omega_k=\bigcup_{j\in J_k} Q_j^k.
\end{equation}
The collection of cubes
\[
S=\{Q_j^k\ :\ \Omega_k\ne\emptyset,\ j\in J_k\}
\]
is sparse, and for all $j$ and $k$, the sets
\[
E(Q_j^k):=Q_j^k\setminus\Omega_{k+1}
\]
satisfy
\begin{equation}\label{eq:CZ-decomposition-6}
\mu(Q_j^k)\le 2\mu(E(Q_j^k)).
\end{equation}
\end{enumerate}
\end{theorem}
\begin{proof}
The proof is analogous to the proof of \cite[Proposition~A.1]{CMP11}.
For the convenience of the reader, we provide the proof in the case of 
$\mu(X)=\infty$. For $\mu(X)<\infty$, the proof is similar.

(a) Let $\Lambda_\lambda$ be the family of dyadic cubes $Q\in\cD$ such that
\begin{equation}\label{eq:CZ-decomposition-7}
\lambda<\frac{1}{\mu(Q)}\int_Q|f(x)|\,d\mu(x).
\end{equation}
Notice that $\Lambda_\lambda$ is nonempty because $\Omega_\lambda\ne\emptyset$.
For each $Q\in\Lambda_\lambda$ there exists a maximal cube 
$Q'\in\Lambda_\lambda$ with $Q\subset Q'$, since
\[
0\le
\frac{1}{\mu(Q)}\int_Q |f(x)|\,d\mu(x)
\le
\frac{1}{\mu(Q)}\int_X |f(x)|\,d\mu(x)\to 0
\quad\mbox{as}\quad \mu(Q)\to\infty.
\]
Let $\{Q_j\}\subset\Lambda_\lambda$ denote the family of such maximal cubes. 
By the maximality, the cubes in $\{Q_j\}$ are pairwise disjoint. If
$\widetilde{Q}_j$ is the dyadic parent of $Q_j$, then 
$Q_j\subset\widetilde{Q}_j$ and $\widetilde{Q}_j$ does not belong to $\{Q_j\}$
in view of the maximality of the cubes in $\{Q_j\}$. Hence, taking into
account Lemma~\ref{le:measure-child-father}, we see that
\[
\lambda
<
\frac{1}{\mu(Q_j)}\int_{Q_j}|f(x)|\,d\mu(x)
\le
\frac{1}{\eps\mu(\widetilde{Q}_j)}\int_{\widetilde{Q}_j}|f(x)|\,d\mu(x)
\le 
\frac{\lambda}{\eps},
\]
which completes the proof of \eqref{eq:CZ-decomposition-2}.

If $x\in\Omega_\lambda$, then it follows from the definition of $M^\cD f$
that there exists a cube $Q\in\cD$ such that $x\in Q$ and 
\eqref{eq:CZ-decomposition-7} is fulfilled. Hence $Q\subset Q_j$ for some $j$.
Therefore, $\Omega_\lambda\subset\bigcup_j Q_j$.

Conversely, since 
\[
\lambda<\frac{1}{\mu(Q_j)}\int_{Q_j}|f(x)|\,d\mu(x),
\]
if $x\in Q_j$, then $(M^\cD f)(x)>\lambda$. This means that 
$x\in\Omega_\lambda$. Therefore, $\bigcup_j Q_j\subset\Omega_\lambda$.
Thus \eqref{eq:CZ-decomposition-1} holds. Part (a) is proved.

(b) Equality \eqref{eq:CZ-decomposition-5} follows from part (a). Since
$\Omega_{k+1}\subset\Omega_k$ and for each fixed $k$, the cubes
$Q_j^k$ are pairwise disjoint, it is clear that the sets $E(Q_j^k)$ are 
pairwise disjoint for all $j$ and $k$. If $Q_j^k\cap Q_i^{k+1}\ne\emptyset$,
then by the maximality of the cubes in $\{Q_j^k\}_{j\in J_k}$ and the 
hypothesis $a>2/\eps$, we have $Q_i^{k+1}\subsetneqq Q_j^k$.
In view of part (a),
\begin{align*}
\mu(Q_j^k\cap\Omega_{k+1})
&=
\sum_{\{i:\ Q_i^{k+1}\subsetneqq Q_j^k\}} 
\mu(Q_i^{k+1})
\\
&\le 
\sum_{\{i:\ Q_i^{k+1}\subsetneqq Q_j^k\}}
\frac{1}{a^{k+1}}\int_{Q_i^{k+1}}|f(x)|\,d\mu(x)
\\
&\le 
\frac{1}{a^{k+1}}\int_{Q_j^k}|f(x)|\,d\mu(x)
\\
&\le 
\frac{1}{a^{k+1}}\cdot\frac{a^k\mu(Q_j^k)}{\eps}=\frac{\mu(Q_j^k)}{a\eps}.
\end{align*}
Then
\begin{align*}
\mu(E(Q_j^k))
&=
\mu(Q_j^k\setminus\Omega_{k+1})=\mu(Q_j^k)-\mu(Q_j^k\cap\Omega_{k+1})
\\
&\ge 
\left(1-\frac{1}{a\eps}\right)\mu(Q_j^k)
>\left(1-\frac{1}{2}\right)\mu(Q_j^k),
\end{align*}
whence $\mu(Q_j^k)\le 2\mu(E(Q_j^k))$ for all $j$ and $k$, which completes the 
proof of \eqref{eq:CZ-decomposition-6}.
\end{proof}
\begin{corollary}\label{co:CZ-corollary}
Let $(X,d,\mu)$ be a space of homogeneous type and $\cD\in\bigcup_{t=1}^K\cD^t$ 
be a dyadic grid on $X$. Suppose $\eps\in(0,1)$ is the same as in 
Lemma~\ref{le:measure-child-father} and $a>2/\eps$. For a non-negative function 
$f\in L^1(X,d,\mu)$ and $k\in\Z$ satisfying \eqref{eq:CZ-decomposition-3}, let
the sets $\Omega_k$ be given by \eqref{eq:CZ-decomposition-4}.
For all $\ell\in\Z_+$ and all $j$ such that 
$Q_j^k\subset\Omega_k$,
\begin{equation}\label{eq:CZ-corollary-1}
\mu(Q_j^k\cap\Omega_{k+\ell})\le \frac{\mu(Q_j^k)}{a^{\ell} \eps}.
\end{equation}
\end{corollary}
\begin{proof}
The proof is analogous to the proof in the Euclidean setting of $\R^n$ given in
\cite[Lemma~2.4]{L17}. Since the cubes of $\Omega_{k+\ell}$ are pairwise 
disjoint and maximal, it follows from Theorem~\ref{th:CZ-decomposition}(a) that
\begin{align*}
\mu(Q_j^k\cap\Omega_{k+\ell}) 
&
=
\sum_{\{i:Q_i^{k+\ell}\subsetneqq Q_j^k\}}\mu(Q_i^{k+\ell})
\\
&< 
\frac{1}{a^{k+\ell}}\sum_{\{i:Q_i^{k+\ell}\subsetneqq Q_j^k\}}
\int_{Q_i^{k+\ell}}f(x)\,d\mu(x)
\\
&\le 
\frac{1}{a^{k+\ell}}\int_{Q_j^k}f(x)\,d\mu(x)
\\
&\le 
\frac{\mu(Q_j^k)}{a^{k+\ell}}\cdot\frac{a^k}{\eps}
=
\frac{\mu(Q_j^k)}{a^{\ell}\eps},
\end{align*}
which completes the proof of \eqref{eq:CZ-corollary-1}.
\end{proof}
The following lemma in the Euclidean setting of $\R^n$ was proved in 
\cite[Lemma~2.6]{L17}.
\begin{lemma}\label{le:estimate-for-dyadic-maximal}
Let $(X,d,\mu)$ be a space of homogeneous type and $\cD\in\bigcup_{t=1}^K\cD^t$ 
be a dyadic grid on $X$. Suppose $\eps\in(0,1)$ is the same as in 
Lemma~\ref{le:measure-child-father} and $a>2/\eps$. For every non-negative 
function $f\in L^1(X,d,\mu)$ there exists a sparse family $S\subset\cD$ 
(depending on $f$) such that for all $x\in X$,
\[
(M^\cD f)(x)
\le 
a\sum_{Q\in S} \left(\frac{1}{\mu(Q)}\int_Q f(y)\,d\mu(y)\right) 
\chi_{E(Q)}(x).
\]
\end{lemma}
\begin{proof}
The proof is, actually, contained in the proof of 
\cite[Theorem~3.1, p.~30]{ACM15}. We reproduce it here for completeness.

Let $\mathbb{K}$ denote the set of all $k\in\Z$ satisfying 
\eqref{eq:CZ-decomposition-3}. Then
\begin{equation}\label{eq:estimate-for-dyadic-maximal-1}
X=\bigcup_{k\in\mathbb{K}}\Omega_k\setminus\Omega_{k+1}.
\end{equation}
Let $S$ be the sparse family given by Theorem~\ref{th:CZ-decomposition}(b).
For $k\in\mathbb{K}$ and a given $x\in\Omega_k\setminus\Omega_{k+1}$, there
exists a cube $Q_j^k\in S$ such that $x\in Q_j^k\setminus\Omega_{k+1}$ and 
\begin{equation}\label{eq:estimate-for-dyadic-maximal-2}
(M^\cD f)(x)\le a^{k+1}<\frac{a}{\mu(Q_j^k)}\int_{Q_j^k} f(y)\,d\mu(y).
\end{equation}
Taking into account that by Theorem~\ref{th:CZ-decomposition}(b),
\begin{equation}\label{eq:estimate-for-dyadic-maximal-3}
\Omega_k\setminus\Omega_{k+1}
=
\left(\bigcup_{j\in J_k}Q_j^k\right)\setminus\Omega_{k+1}
=
\bigcup_{j\in J_k}E(Q_j^k),
\end{equation}
we obtain from \eqref{eq:estimate-for-dyadic-maximal-1}--\eqref{eq:estimate-for-dyadic-maximal-3}
for all $x\in X$,
\begin{align*}
(M^\cD f)(x)
&=
\sum_{k\in\mathbb{K}} (M^\cD f)(x)\chi_{\Omega_k\setminus\Omega_{k+1}} (x)
\\
&\le 
\sum_{k\in\mathbb{K}}\sum_{j\in J_k}
\left(\frac{a}{\mu(Q_j^k)}\int_{Q_j^k} f(y)\,d\mu(y)\right)\chi_{E(Q_j^k)}(x)
\\
&=
a\sum_{Q\in S}\left(\frac{1}{\mu(Q)}\int_Q f(y)\,d\mu(y)\right) 
\chi_{E(Q)}(x),
\end{align*}
which completes the proof.
\end{proof}
\section{Proof of part (a) of Theorem~\ref{th:main}.}
\label{sec:proof-necessity}
The scheme of the proof is borrowed from the proof of the necessity portion of
\cite[Theorem~3.1]{L17}.

For a bounded sublinear operator on a Banach function space $\cE'(X,d,\mu)$,
let $\|T\|_{\cB(\cE')}$ denote its norm.

Fix $\cD\in\bigcup_{t=1}^K\cD^t$. It follows from the boundedness of the 
Hardy-Littlewood maximal operator $M$ on $\cE'(X,d,\mu)$ in view of
Theorem~\ref{th:HK-pointwise} and the lattice property (axiom (A2) in the 
definition of a Banach function space) that the dyadic maximal operator
$M^\cD$ is bounded on the space $\cE'(X,d,\mu)$ and
\begin{equation}\label{eq:proof-a-1}
\|M^\cD\|_{\cB(\cE')}\le C_{HK}(X)\|M\|_{\cB(\cE')}.
\end{equation}
Let $g\in L_+^0(X,d,\mu)$ with $\|g\|_{\cE'}\le 1$. Using an idea of
Rubio de Francia \cite{RF84} (see also \cite[Section~2.1]{CMP11}), put
\[
(Rg)(x):=\sum_{k=0}^\infty
\frac{\big((M^\cD)^kg\big)(x)}{\big(2\|M^\cD\|_{\cB(\cE')}\big)^k},
\quad x\in X,
\]
where $(M^\cD)^k$ denotes the $k$-th iteration of $M^\cD$ and $(M^\cD)^0g:=g$.
Then
\begin{equation}\label{eq:proof-a-2}
\|Rg\|_{\cE'}\le 2
\end{equation}
and
\begin{equation}\label{eq:proof-a-3}
g(x)\le (Rg)(x)
\quad\mbox{for a.e.}\quad x\in X.
\end{equation}
Since $M^\cD$ is sublinear, we have
\begin{align*}
(M^\cD Rg)(x)
&\le 
\sum_{k=0}^\infty
\frac{\big((M^\cD)^{k+1}g\big)(x)}{\big(2\|M^\cD\|_{\cB(\cE')}\big)^k}
\\
&=
2\|M^\cD\|_{\cB(\cE')}
\sum_{k=0}^\infty
\frac{\big((M^\cD)^{k+1}g\big)(x)}{\big(2\|M^\cD\|_{\cB(\cE')}\big)^{k+1}}
\\
&\le 
2\|M^\cD\|_{\cB(\cE')}
\sum_{k=0}^\infty
\frac{\big((M^\cD)^{k}g\big)(x)}{\big(2\|M^\cD\|_{\cB(\cE')}\big)^{k}}
\\
&=
2\|M^\cD\|_{\cB(\cE')}(Rg)(x),
\end{align*}
whence $Rg\in A_1^\cD$ with
\begin{equation}\label{eq:proof-a-4}
[Rg]_{A_1^\cD}\le 2\|M^\cD\|_{\cB(\cE')}.
\end{equation}
Let the constants $C_{\cD^t}\ge 1$ be defined for each $t\in\{1,\dots,K\}$
by \eqref{eq:constant-CD}.
Take $\eta$ and $\gamma$ such that
\[
0<\eta\le 
\left(
\left(\max_{1\le t\le K}C_{\cD^t}^2\right) 
K C_{HK}(X)\|M\|_{\cB(\cE')}\right)^{-1},
\quad
\gamma=\frac{\eta}{1+\eta}.
\]
Inequalities \eqref{eq:proof-a-4} and \eqref{eq:proof-a-1} imply that
\begin{equation}\label{eq:proof-a-5}
[Rg]_{A_1^\cD}\le 2C_{HK}(X)\|M\|_{\cB(\cE')},
\end{equation}
whence $\eta$ satisfies \eqref{eq:RHI-1}. Since $Rg\in A_1^\cD$, it follows 
from  Theorem~\ref{th:RHI-corollary} and inequality \eqref{eq:proof-a-5} that, 
for every cube $Q\in\cD$ and every measurable subset $G_Q\subset Q$, one has
\begin{align}
\int_{G_Q}(Rg)(x)\,d\mu(x)
&\le 
2^{1-\gamma} C_\cD[Rg]_{A_1^\cD}
\left(\frac{\mu(G_Q)}{\mu(Q)}\right)^\gamma
\int_Q(Rg)(x)\,d\mu(x)
\nonumber\\
&\le
\frac{C}{2}\left(\frac{\mu(G_Q)}{\mu(Q)}\right)^\gamma
\int_Q(Rg)(x)\,d\mu(x),
\label{eq:proof-a-6}
\end{align}
where
\[
C:=2^{3-\gamma}\left(\max_{1\le t\le K}C_{\cD^t}\right) 
C_{HK}(X)\|M\|_{\cB(\cE')}.
\]
Taking into account inequalities \eqref{eq:proof-a-3},
\eqref{eq:proof-a-6}, H\"older's inequality for Banach
function spaces (see \cite[Chap.~1, Theorem~2.4]{BS88}), and inequality
\eqref{eq:proof-a-2}, we deduce that, for every finite
sparse family $S\subset\cD$, every collection of non-negative numbers
$\{\alpha_Q\}_{Q\in S}$, every collection of pairwise disjoint
measurable subsets $G_Q\subset Q$, and every $g\in L_+^0(X,d,\mu)$ satisfying 
$\|g\|_{\cE'}\le 1$, one has
\begin{align*}
\int_X 
\Bigg( &\sum_{Q\in S}\alpha_Q\chi_{G_Q}(x)\Bigg)g(x)\,d\mu(x)
\\
&\le 
\sum_{Q\in S}\alpha_Q \int_{G_Q}g(x)\,d\mu(x)
\\
&\le 
\sum_{Q\in S}\alpha_Q \int_{G_Q} (Rg)(x)\,d\mu(x)
\\
&\le 
\frac{C}{2}\sum_{Q\in S}\alpha_Q
\left(\frac{\mu(G_Q)}{\mu(Q)}\right)^\gamma \int_Q(Rg)(x)\,d\mu(x)
\\
&\le 
\frac{C}{2}\left(\max_{Q\in S}\frac{\mu(G_Q)}{\mu(Q)}\right)^\gamma
\int_X\left(\sum_{Q\in S}\alpha_Q\chi_Q(x)\right)(Rg)(x)\,d\mu(x)
\\
&\le 
\frac{C}{2}\left(\max_{Q\in S}\frac{\mu(G_Q)}{\mu(Q)}\right)^\gamma
\Bigg\|\sum_{Q\in S}\alpha_Q\chi_Q\Bigg\|_{\cE}\|Rg\|_{\cE'}
\\
&\le
C\left(\max_{Q\in S}\frac{\mu(G_Q)}{\mu(Q)}\right)^\gamma
\Bigg\|\sum_{Q\in S}\alpha_Q\chi_Q\Bigg\|_{\cE}.
\end{align*}
Then, in view of the Lorentz-Luxemburg theorem (see 
\cite[Chap.~1, Theorem~2.7]{BS88}),
\begin{align*}
&
\Bigg\|
\sum_{Q\in S}\alpha_Q\chi_{G_Q}\Bigg\|_{\cE}
=
\Bigg\| \sum_{Q\in S}\alpha_Q\chi_{G_Q}\Bigg\|_{\cE''}
\\
&\quad=
\sup\left\{\int_X\Bigg(\sum_{Q\in S}\alpha_Q\chi_{G_Q}(x)\Bigg)g(x)\,d\mu(x):
g\in L_+^0(X,d,\mu), \ \|g\|_{\cE'}\le 1
\right\}
\\
&\quad\le 
C\left(\max_{Q\in S}\frac{\mu(G_Q)}{\mu(Q)}\right)^\gamma
\Bigg\|\sum_{Q\in S}\alpha_Q\chi_Q\Bigg\|_{\cE},
\end{align*}
that is, the space $\cE(X,d,\mu)$ satisfies the condition $\cA_\infty$, which 
completes the proof of part (a) of Theorem~\ref{th:main}.
\qed
\section{Proof of part (b) of Theorem~\ref{th:main}.}
\label{sec:proof-sufficiency}
We follow the proof of the sufficiency portion of the proof of
\cite[Theorem~3.1]{L17}. Let $\eps\in(0,1)$ be the same as in 
Lemma~\ref{le:measure-child-father}. Take $a>2/\eps$. Assume that 
$f\in L^1(X,d,\mu)\cap\cE'(X,d,\mu)$ is a nonnegative function and fix any dyadic 
grid $\cD\in\bigcup_{t=1}^k\cD^t$. By Lemma~\ref{le:estimate-for-dyadic-maximal}, 
there exists a sparse family $S\subset\cD$ (not necessarily finite) such that
for all $x\in X$,
\begin{equation}\label{eq:proof-b-1}
(M^\cD f)(x)
\le 
a\sum_{Q\in S}\left(\frac{1}{\mu(Q)}\int_Q f(y)\,d\mu(y)\right)\chi_{E(Q)}(x).
\end{equation}
For every subfamily $S'\subset S$, put
\[
(A_{S'}f)(x)=
\sum_{Q\in S'}\left(\frac{1}{\mu(Q)}\int_Q f(y)\,d\mu(y)\right)\chi_{E(Q)}(x).
\]
Let $\{S_t\}_{t\in \N}$ be a sequence of subfamilies of $S$ such that
each subfamily $S_t$ is finite, $S_t\subset S_n$ if $t<n$, and 
$A_{S_t}f\uparrow A_S f$ a.e. on $X$ as $t\to\infty$. By the Fatou property
(axiom (A3) in the definition of a Banach function space),
\begin{equation}\label{eq:proof-b-2}
\lim_{t\to\infty}\|A_{S_t}f\|_{\cE'}=\|A_Sf\|_{\cE'}.
\end{equation}
By the Fubini theorem, for every $g\in\cE(X,d,\mu)$ and every $t\in\N$,  
one has
\begin{align}
\int_X & (A_{S_t}f) (x)g(x)\,d\mu(x)
\nonumber\\
&=
\int_X\int_X\sum_{Q\in S_t}\frac{1}{\mu(Q)}f(y)\chi_Q(y)\chi_{E(Q)}(x)g(x)
\,d\mu(y)\,d\mu(x)
\nonumber\\
&=
\int_X\int_X\sum_{Q\in S_t}\frac{1}{\mu(Q)}f(x)\chi_Q(x)\chi_{E(Q)}(y)g(y)
\,d\mu(y)\,d\mu(x)
\nonumber\\
&=
\int_X\sum_{Q\in S_t}\left(\frac{1}{\mu(Q)}\int_{E(Q)}g(y)\,d\mu(y)\right)
\chi_Q(x)f(x)\,d\mu(x)
\nonumber\\
&=
\int_X f(x)(A_{S_t}^*g)(x)\,d\mu(x),
\label{eq:proof-b-3}
\end{align}
where
\[
(A_{S_t}^*g)(x)
:=
\sum_{Q\in S_t}\left(\frac{1}{\mu(Q)}\int_{E(Q)}g(y)\,d\mu(y)\right)\chi_Q(x),
\quad
x\in X.
\]
It follows from \eqref{eq:proof-b-3} and H\"older's inequality for
Banach function spaces (see \cite[Chap.~1, Theorem~2.4]{BS88}) that
\begin{equation}\label{eq:proof-b-4}
\left|\int_X (A_{S_t}f)(x)g(x)\,d\mu(x)\right|
\leq
\|A_{S_t}^*g\|_\cE\|f\|_{\cE'}.
\end{equation}
Let $C,\gamma>0$ be as in Definition~\ref{def:A-infinity}. Since
$a>2/\eps>2$, there exists $\nu\in\N$ such that
\begin{equation}\label{eq:proof-b-5}
C\eps^{-\gamma}\sum_{\ell=\nu}^\infty a^{-\ell\gamma}\le\frac{1}{2}.
\end{equation}
For $Q\in S_t$, let
\[
\alpha_Q:=\frac{1}{\mu(Q)}\int_{E(Q)}|g(x)|\,d\mu(x).
\]
Then for all $x\in X$,
\begin{align}
|(A_{S_t}^*g)(x)| 
&\leq 
\sum_{Q\in S_t}\alpha_Q\chi_Q(x)
\nonumber\\
&=
\sum_{\{j,k:Q_j^k\in S_t\}}\alpha_{Q_j^k}\chi_{Q_j^k}(x)
\nonumber\\
&=
\sum_{\{j,k:Q_j^k\in S_t\}}\alpha_{Q_j^k}\chi_{Q_j^k\setminus\Omega_{k+\nu}}(x)
+
\sum_{\{j,k:Q_j^k\in S_t\}}\alpha_{Q_j^k}\chi_{Q_j^k\cap\Omega_{k+\nu}}(x)
\nonumber\\
&=:\Sigma_1(x)+\Sigma_2(x),
\label{eq:proof-b-6}
\end{align}
where the sets $\Omega_k$ are defined by \eqref{eq:CZ-decomposition-4}
for all $k\in\Z$ satisfying \eqref{eq:CZ-decomposition-3}.

Let $\mathbb{K}$ be the set of all those $k\in\Z$ that satisfy 
\eqref{eq:CZ-decomposition-3}. It is easy to see that for $k\in\mathbb{K}$
and $\nu\in\N$,
\begin{equation}\label{eq:proof-b-7}
\Omega_k\setminus\Omega_{k+\nu}
\subset
\bigcup_{i=0}^{\nu-1}\Omega_{k+i}\setminus\Omega_{k+i+1}.
\end{equation}
It is also easy to see that if $k\in\mathbb{K}$ and $x\in Q_j^k$, then
\begin{equation}\label{eq:proof-b-8}
\alpha_{Q_j^k}
\le
\frac{1}{\mu(Q_j^k)}\int_{Q_j^k}|g(x)|\,d\mu(x)
\leq 
(M^\cD g)(x).
\end{equation}
Combining \eqref{eq:proof-b-7} and \eqref{eq:proof-b-8}, we get for $x\in X$,
\begin{align}
\Sigma_1(x) 
&=
\sum_{\{j,k:Q_j^k\in S_t\}}\alpha_{Q_j^k}
\chi_{Q_j^k\setminus\Omega_{k+\nu}}(x)
\nonumber\\
&\le 
(M^\cD g)(x)\sum_{k\in \mathbb{K}}
\chi_{\Omega_k\setminus\Omega_{k+\nu}}(x)
\nonumber\\
&\le 
(M^\cD g)(x)\sum_{i=0}^{\nu-1}
\sum_{k\in \mathbb{K}}\chi_{\Omega_{k+i}\setminus\Omega_{k+i+1}}(x)
\nonumber\\
&=\nu(M^\cD g)(x).
\label{eq:proof-b-9}
\end{align}
On the other hand, for $x\in X$, we have
\begin{align}
\Sigma_2(x) 
&=
\sum_{\{j,k:Q_j^k\in S_t\}}
\alpha_{Q_j^k}\chi_{Q_j^k\cap\Omega_{k+\nu}}(x)
\nonumber\\
&=
\sum_{\{j,k:Q_j^k\in S_t\}}
\alpha_{Q_j^k}\sum_{\ell=\nu}^\infty
\chi_{Q_j^k\cap(\Omega_{k+\ell}\setminus\Omega_{k+\ell+1})}(x)
\nonumber\\
&=
\sum_{\ell=\nu}^\infty
\sum_{\{j,k:Q_j^k\in S_t\}}
\alpha_{Q_j^k}
\chi_{Q_j^k\cap(\Omega_{k+\ell}\setminus\Omega_{k+\ell+1})}(x).
\label{eq:proof-b-10}
\end{align}
Since $S_t$ is a finite sparse family, applying inequality \eqref{eq:A-infty}
of Definition~\ref{def:A-infinity}, we obtain for all $\ell\ge\nu$,
\begin{align}
&
\left\|
\sum_{\{j,k:Q_j^k\in S_t\}} 
\alpha_{Q_j^k}\chi_{Q_j^k\cap(\Omega_{k+\ell}\setminus\Omega_{k+\ell+1})}
\right\|_\cE
\nonumber\\
&\quad\quad\le 
C\left(
\max_{\{j,k:Q_j^k\in S_t\}}
\frac{\mu\big(Q_j^k\cap(\Omega_{k+\ell}\setminus\Omega_{k+\ell+1})\big)}
{\mu(Q_j^k)}
\right)^\gamma \|A_{S_t}^*g\|_\cE.
\label{eq:proof-b-11}
\end{align}
By Corollary~\ref{co:CZ-corollary}, we get for all $\ell\ge\nu$,
\begin{align}
\max_{\{j,k:Q_j^k\in S_t\}}
\frac{\mu\big(Q_j^k\cap(\Omega_{k+\ell}\setminus\Omega_{k+\ell+1})\big)}
{\mu(Q_j^k)}
&\leq 
\max_{\{j,k:Q_j^k\in S_t\}}
\frac{\mu(Q_j^k\cap\Omega_{k+\ell})}{\mu(Q_j^k)}
\nonumber\\
&\le 
\max_{\{j,k:Q_j^k\in S_t\}}
\frac{\mu(Q_j^k)}{a^\ell\eps\mu(Q_j^k)}=\frac{1}{a^\ell\eps}.
\label{eq:proof-b-12}
\end{align}
It follows from \eqref{eq:proof-b-10}--\eqref{eq:proof-b-12} and 
\eqref{eq:proof-b-5} that
\begin{align}
\left\|\Sigma_2\right\|_\cE
&\leq
\sum_{\ell=\nu}^\infty\left\|
\sum_{\{j,k:Q_j^k\in S_t\}} 
\alpha_{Q_j^k}\chi_{Q_j^k\cap(\Omega_{k+\ell}\setminus\Omega_{k+\ell+1})}
\right\|_\cE
\nonumber\\
&\leq
C\sum_{\ell=\nu}^\infty\left(\frac{1}{a^\ell\eps}\right)^\gamma\|A_{S_t}^*g\|_\cE
\le\frac{1}{2}\|A_{S_t}^* g\|_\cE.
\label{eq:proof-b-13}
\end{align}
Combining inequalities \eqref{eq:proof-b-6}, \eqref{eq:proof-b-9}, and
\eqref{eq:proof-b-13}, we arrive at
\[
\|A_{S_t}^*g\|_\cE\le\nu\|M^\cD g\|_\cE+\frac{1}{2}\|A_{S_t}^*g\|_\cE.
\]
It follows from this inequality, Theorem~\ref{th:HK-pointwise}, and
the boundedness of the Hardy-Littlewood maximal operator on $\cE(X,d,\mu)$
that for all finite sparse families $S_t\subset S$ and all
$g\in\cE(X,d,\mu)$, one has
\begin{align}
\|A_{S_t}^*g\|_\cE 
&\leq 
2\nu\|M^\cD g\|_\cE
\leq
2\nu C_{HK}(X)\|Mg\|_\cE
\nonumber\\
& \leq
2\nu C_{HK}(X)\|M\|_{\cB(\cE)}\|g\|_\cE.
\label{eq:proof-b-14}
\end{align}
Combining \eqref{eq:proof-b-4} and \eqref{eq:proof-b-14} with 
\cite[Chap.~1, Lemma~2.8]{BS88}, we see that
\begin{align*}
\|A_{S_t}f\|_{\cE'}
&=
\sup\left\{\left|\int_X (A_{S_t}f)(x)g(x)\,d\mu(x)\right|\ :\ g\in\cE(X,d,\mu),\ 
\|g\|_\cE\le 1\right\}
\\
&\leq
2\nu C_{HK}(X)\|M\|_{\cB(\cE)}\|f\|_{\cE'}
\end{align*}
for all $t\in\N$. Passing in this inequality to the limit as $t\to\infty$ and 
taking into account \eqref{eq:proof-b-2}, we get
\[
\|A_Sf\|_{\cE'}\le 2\nu C_{HK}(X)\|M\|_{\cB(\cE)}\|f\|_{\cE'}.
\]
It follows from this inequality and inequality \eqref{eq:proof-b-1} that
\begin{equation}\label{eq:proof-b-15}
\|M^\cD f\|_{\cE'}\le 2a\nu C_{HK}(X)\|M\|_{\cB(\cE)}\|f\|_{\cE'}.
\end{equation}
for every dyadic grid $\cD\in\bigcup_{t=1}^K\cD^t$. In turn, inequality
\eqref{eq:proof-b-15} and Theorem~\ref{th:HK-pointwise} imply that
\begin{equation}\label{eq:proof-b-16}
\|Mf\|_{\cE'}\le 2a\nu KC_{HK}^2(X)\|M\|_{\cB(\cE)}\|f\|_{\cE'}
\end{equation}
for every nonnegative function $f\in L^1(X,d,\mu)\cap\cE'(X,d,\mu)$.

Now let $f\in\cE'(X,d,\mu)$ be an arbitrary complex-valued function. Since
$X$ is $\sigma$-finite, there are measurable sets $\{A_n\}_{n\in\N}$ such that
$\mu(A_n)<\infty$ for all $n\in\N$, $A_i\subset A_j$ for $i<j$, and 
$\bigcup_{n\in\N}A_n=X$. Let $f_n=|f|\chi_{A_n}$ for $n\in\N$. Then
$f_n\in L^1(X,d,\mu)\cap\cE'(X,d,\mu)$ for all $n\in\N$ in view of axiom (A5)
in the definition of a Banach function space. By \eqref{eq:proof-b-16},
for all $n\in\N$,
\begin{equation}\label{eq:proof-b-17}
\|Mf_n\|_{\cE'}\le 2a\nu KC_{HK}^2(X)\|M\|_{\cB(\cE)}\|f_n\|_{\cE'}.
\end{equation}
Since $f_n\uparrow|f|$ a.e., we have $Mf_n\uparrow Mf$ a.e.
(cf. \cite[Lemma~2.2]{CDF09}). Passing to the limit in inequality
\eqref{eq:proof-b-17}, we conclude from the Fatou property that inequality
\eqref{eq:proof-b-16} holds for all $f\in\cE'(X,d,\mu)$. Thus, the 
Hardy-Littlewood maximal operator $M$ is bounded on the space $\cE'(X,d,\mu)$
whenever it is bounded on the space $\cE(X,d,\mu)$ and the latter space 
satisfies the condition $\cA_\infty$.
\qed

\end{document}